\documentclass[a4paper,10pt]{amsart}

\usepackage[cp850]{inputenc}
\usepackage{latexsym}
\usepackage{amsfonts}
\usepackage{amsmath}
\usepackage{amssymb,amsxtra,amscd}
\usepackage{mathrsfs}
\usepackage{hyperref}

\newcommand{\R}{\mathbb{R}}
\newcommand{\C}{\mathbb{C}}

\renewcommand{\Re}{\mbox{Re}\,}
\renewcommand{\Im}{\mbox{Im}\,}

\newtheorem{theo}{Theorem}
\newtheorem{lem}[theo]{Lemma}

\newtheorem{prop}[theo]{Proposition}

\newtheorem{rem}[theo]{Remark}

\newtheorem{example}[theo]{Example}

\title[FHC for weighted composition semigroups]{A remark on the Frequent Hypercyclicity Criterion for weighted composition semigroups and an application to the linear von Foerster-Lasota equation}
\author{T.\ Kalmes}
\thanks{\begin{sc} \hspace{-.5cm}Technische Universit\"at Chemnitz,
Fakult\"at f\"ur Mathematik, 09107 Chemnitz, Germany,\end{sc} \begin{it}E-mail address:\end{it} thomas.kalmes@mathematik.tu-chemnitz.de}

\begin{document}

\begin{abstract}
We generalize a result for the translation $C_0$-semigroup on $L^p(\R_+,\mu)$ about the equivalence of being chaotic and satisfying the Frequent Hypercyclicity criterion due to Mangino and Peris \cite{MaPe11} to certain weighted composition $C_0$-semigroups. Such $C_0$-semigroups appear in a natural way when dealing with initial value problems for linear first order partial differential operators. We apply our result to the linear von Foerster-Lasota equation arising in mathematical biology. Weighted composition $C_0$-semigroups on Sobolev spaces are also considered.
\end{abstract}

\maketitle

\section{Introduction}

The purpose of this article is to generalize a result about the dynamical behaviour of the translation $C_0$-semigroup on $L^p(\R_+,\mu)$ to a larger class of certain weighted composition $C_0$-semigroups on $L^p(\Omega,\mu)$, where $\mu$ is a Borel measure on an open subset $\Omega$ of $\R$ admitting a strictly positive Lebesgue density $\rho$. Recall that a $C_0$-semigroup $T$ on a separable Banach space $X$ is called {\it chaotic} if $T$ is hypercyclic, i.e.\ if there is $x\in X$ such that $\{T(t)x;\,t\geq 0\}$ is dense in $X$, and if the set of periodic points, i.e.\ $\{x\in X;\,\exists t>0: T(t)x=x\}$, is dense in $X$.

Adapting the notion of frequent hypercyclicity, which was introduced by Bayart and Grivaux in 2005 for a single operator \cite{BaGr05}, a $C_0$-semigroup $T$ is called {\it frequently hypercyclic} if there is $x\in X$ such that for every non-empty, open $U\subseteq X$
\[\underline{\mbox{dens}}\{t\geq 0;\,T(t)x\in U\}>0,\]
where the lower density $\underline{\mbox{dens}}(M)$ of a Lebesgue measurable subset $M$ of $[0,\infty)$ is defined as
\[\underline{\mbox{dens}}(M):=\liminf_{t\rightarrow\infty}\frac{\lambda(M\cap [0,t])}{t},\]
where $\lambda$ denotes Lebesgue measure. In \cite{MaPe11}, Mangino and Peris gave a sufficient condition for a $C_0$-semigroup to be frequently hypercyclic, the so-called Frequent Hypercyclicity criterion. It was shown by Murillo-Arcila and Peris in \cite{MuPe14} that the hypotheses of the Frequent Hypercyclicity criterion even imply the existence of a strongly mixing Borel probability measure with full support for the $C_0$-semigroup $T$. Moreover, by \cite[Proposition 2.6]{MaPe11}, for every $t_0>0$, the operator $T(t_0)$ is chaotic whenever the $C_0$-semigroup $T$ satisfies the Frequent Hypercyclicity criterion. In particular, $T$ is then chaotic. By \cite[Proposition 3.3]{MaPe11}, the translation $C_0$-semigroup on $L^p(\R_+,\mu)$, defined by $T(t)f(x)=f(x+t)$, satisfies the Frequent Hypercyclicity criterion if and only if the translation $C_0$-semigroup is chaotic.\\

For $\Omega\subseteq\R$ open and a Borel measure $\mu$ on $\Omega$ admitting a strictly positive Lebesgue density $\rho$ we consider $C_0$-semigroups $T$ on
$L^p(\Omega,\mu), 1\leq p<\infty,$ of the form
\[T(t)f(x)=h_t(x) f(\varphi(t,x)),\]
where $\varphi$ is the solution semiflow of an ordinary differential equation
\[\dot{x}=F(x)\]
in $\Omega$ and
\[h_t(x)=\exp\big(\int_0^t h(\varphi(s,x))ds\big)\]
with $h\in C(\Omega)$. Such $C_0$-semigroups appear in a natural way when dealing with initial value problems for linear first order partial differential operators. Under mild assumptions on $F$ and $h$, in section 2 we prove that, as in the case of the translation $C_0$-semigroup, these $C_0$-semigroups satisfy the Frequent Hypercyclicity criterion if and only if they are chaotic.

As an application of this result we show that for the von Foerster-Lasota $C_0$-semigroup on $L^p(0,1)$ the properties of being hypercyclic, chaotic, or frequently hypercyclic are equivalent.

In section 3, we also consider weighted composition $C_0$-semigroups on the closed subspace
\[W^{1,p}_*[a,b]=\{f\in W^{1,p}[a,b];\,f(a)=0\}\]
of the Sobolev space $W^{1,p}[a,b]$ over the compact interval $[a,b]$, and we prove that the same equivalence between chaos and the Frequent Hypercyclicity criterion holds on $W^{1,p}_*[a,b]$. It should be mentioned that such semigroups are never hypercyclic in the whole space $W^{1,p}[a,b]$, as shown in \cite{ArKaMa13}.

\section{The Frequent Hypercyclicity Criterion applied to weighted composition $C_0$-semigroups}

A sufficient condition for a $C_0$-semigroup $T$ on a separable Banach space $X$ to be frequently hypercyclic, the so called Frequent Hypercyclicity criterion, is the following result due to Mangino and Peris, see \cite[Theorem 2.2]{MaPe11}.

\begin{theo}\label{FHC}
Let $T$ be a $C_0$-semigroup on a separable Banach space $X$. If there exist $X_0\subset X$ dense in $X$ and maps $S_t:X_0\rightarrow X, t\geq 0,$ such that
\begin{itemize}
\item[i)] $T(t)S_tx=x, T(t)S_rx=S_{r-t}x$ for all $r>t>0$ and $x\in X_0$,
\item[ii)] $t\mapsto T(t)x$ is Pettis integrable in $[0,\infty)$ for all $x\in X_0$,
\item[iii)] $t\mapsto S_tx$ is Pettis integrable in $[0,\infty)$ for all $x\in X_0$,
\end{itemize}
then $T$ is frequently hypercyclic.
\end{theo}

Let $\Omega\subseteq\R$ be open and let $F:\Omega\rightarrow\R$ be a $C^1$-function. Hence, for every $x_0\in\Omega$ there is a unique solution $\varphi(\cdot,x_0)$ of the initial value problem
\[\dot{x}=F(x),\; x(0)=x_0.\]
Denoting its maximal domain of definition by $J(x_0)$ it is well-known that $J(x_0)$ is an open interval containing $0$. We make the general assumption that $\Omega$ is {\it forward invariant under} $F$, i.e.\
$[0,\infty)\subset J(x_0)$ for every $x_0\in\Omega$, that is $\varphi:[0,\infty)\rightarrow\Omega$. This is true, for example, if $\Omega=(a,b)$ is a bounded interval and if $F$ can be extended to a $C^1$-function defined on a neighborhood of $[a,b]$ such that $F(a)\geq 0$ and $F(b)\leq 0$ (cf.\ \cite[Corollary 16.10]{Amann}).

From the uniqueness of the solution it follows that $\varphi(t,\cdot)$ is injective for
every $t\geq 0$ and $t+s\in J(x)$ whenever $s\in J(x), t\in J(\varphi(s,x))$ and then $\varphi(t+s,x)=\varphi(t,\varphi(s,x))$. Moreover, for every $t\geq 0$ the set $\varphi(t,\Omega)$ is open and for $x\in\varphi(t,\Omega)$
we have $[-t,\infty)\subset J(x)$ as well as $\varphi(-s,x)=\varphi(s,\cdot)^{-1}(x)$ for all $s\in [0,t]$. Since $F$ is a $C^1$-function it is well-known that the same is true for
$\varphi(t,\cdot)$ on $\Omega$ and $\varphi(-t,\cdot)$ on $\varphi(t,\Omega)$ for every $t\geq 0$. We denote their derivatives by $\partial_2\varphi(t,x)$ and $\partial_2\varphi(-t,x)$, respectively. We then have $\partial_2\varphi(-t,\varphi(t,x))\partial_2\varphi(t,x)=1=\partial_2\varphi(t,\varphi(-t,x))\partial_2\varphi(-t,x)$.

Because $F\in C^1(\Omega)$ it is well known that for $x\in\Omega$ and $t\in J(x)$ there is a neighborhood $U$ of $x$ in $\Omega$ such that $t\in J(y)$ for every $y\in U$ and that $\varphi(t,\cdot)$ is differentiable in $U$. We denote its derivative by $\partial_2\varphi(t,\cdot)$. We then have
\begin{equation}\label{positive}
\forall\,t\geq 0:\,\partial_2\varphi(t,x)=\exp(\int_0^tF'(\varphi(s,x))ds)>0
\end{equation}
and
\[\forall\,r\geq 0, x\in\varphi(r,\Omega), t\in[0,r]:\,\partial_2\varphi(-t,x)=\exp(-\int_{-t}^0 F'(\varphi(s,x)ds),\]
see e.g.\ \cite[Proposition 11]{ArKaMa13}.

Moreover, let $h\in C(\Omega)$ and define for $t\ge 0$
\[h_t:\Omega\rightarrow\C, h_t(x)=\exp(\int_0^t h(\varphi(s,x))ds).\]
For $1\leq p<\infty$ and a measurable function $\rho:\Omega\rightarrow (0,\infty)$ let $L^p_\rho(\Omega)$ be as usual
the Lebesgue space of $p$-integrable functions with respect to the Borel measure $\rho d\lambda$, where $\lambda$
denotes Lebesgue measure. If $\Omega$ is forward invariant under $F$ the operators
\[T(t):L^p_\rho(\Omega)\rightarrow L^p_\rho(\Omega), (T(t)f)(x):=h_t(x)f(\varphi(t,x))\;(t\geq 0)\]
are well-defined continuous linear operators defining a $C_0$-semigroup $T_{F,h}$ on $L^p_\rho(\Omega)$ if $\rho$ is
$p$-{\it admissible for} $F$ {\it and} $h$, i.e.\ if there are constants $M\geq 1$, $\omega\in\R$ with
\[\forall\,t\geq 0,x\in\Omega:\,|h_t(x)|^p\rho(x)\leq M e^{\omega t}\rho(\varphi(t,x))\exp(\int_0^t F'(\varphi(s,x))ds),\]
(see \cite{ArKaMa13}). Because $|h_t(x)|^p=\exp(p\int_0^t\Re h(\varphi(s,x))ds)$ it follows that $\rho=1$ is 
$p$-admissible for any $p$ if $\Re h$ is bounded above and $F'$ is bounded below, i.e.\ in this case the above operators
define a $C_0$-semigroup $T_{F,h}$ on the standard Lebesgue spaces $L^p(\Omega)$. Under mild additional assumptions on $F$ and $h$ the generator of  this $C_0$-semigroup is given by the first order differential operator $Af=F f'+hf$ on a suitable subspace of $L^p(\Omega)$ (see \cite[Theorem 15]{ArKaMa13}).

If $\rho$ is $p$-admissible for $F$ and $h$ it follows from \cite[Lemma 7]{Ka09} that $\chi_I\in L^p_\rho(\Omega)$ for each compact interval $I\subset\Omega\backslash\{F=0\}$. The main result of this article is the following theorem.

\begin{theo}\label{equivalence}
Let $\Omega\subseteq\R$ be an open interval which is forward invariant under $F\in C^1(\Omega)$, and let $h\in C(\Omega)$ be such that $F'$ and $\Re h$ are bounded and
\begin{itemize}
	\item[a)] There is $\gamma\in\R$ such that $h(x)=\gamma$ for all $x\in\{F=0\}$.
	\item[b)] With $\alpha:=\inf\Omega$ and $\omega:=\sup\Omega$ the function
	\[\Omega\rightarrow\C,y\mapsto\frac{\Im h(y)}{F(y)}\]
	belongs to $L^1((\alpha,\beta))$ for all $\beta\in\Omega$ or to $L^1((\beta,\omega))$ for all $\beta\in\Omega$.
\end{itemize}
Moreover, let $\rho$ be $p$-admissible for $F$ and $h$. Then the following are equivalent.
	\begin{itemize}
		\item[i)] $T_{F,h}$ is chaotic in $L^p_\rho(\Omega)$.
		\item[ii)] $X_0:=\mbox{span}\{\chi_I;\,I\subseteq\Omega\backslash\{F=0\}\mbox{ compact interval}\}$ is a dense subspace of $L^p_\rho(\Omega)$ and $T_{F,h}$ satisfies the Frequent Hypercyclicity criterion with $X_0$
		and linear mappings $S_t:X_0\rightarrow L^p_\rho(\Omega)$ which are given by
		\[S_t\chi_I:=\chi_{\varphi(t,\Omega)}(\cdot)h_t(\varphi(-t,\cdot))^{-1}\chi_I(\varphi(-t,\cdot))\]
		for a closed interval $I\subset\Omega\backslash\{F=0\}$.
		\item[iii)] $T_{F,h}$ satisfies the Frequent Hypercyclicity criterion.
	\end{itemize}
\end{theo}

\begin{rem}
\begin{rm}
It is shown in \cite{Ka14} that under the hypothesis of Theorem \ref{equivalence} the $C_0$-semigroup $T_{F,h}$ is chaotic if and only if $\lambda(\{F=0\})=0$ and for each connected component $C$ of $\Omega\backslash\{F=0\}$
\[\int_C \exp(-p\int_x^w \frac{\Re h(y)}{F(y)}dy)\rho(w)d\lambda(w)<\infty\]
for some/all $x\in C$.
\end{rm}
\end{rem}

For the special case of the translation semigroup on $L^p_\rho(\R_+)$, i.e.\ $F(x)=1$ and $h=0$ so that $\varphi(t,x)=x+t$ and $h_t=1$, the above result is \cite[Proposition 3.3]{MaPe11}. The proof of Theorem \ref{equivalence} is along the same lines as the proof of \cite[Proposition 3.3]{MaPe11}, although in the general case, it is more involved. We begin by providing the following auxiliary result. 

\begin{lem}\label{aux}
Let $\Omega\subseteq\R$ be open and forward invariant under $F\in C^1(\Omega)$. Then, for every compact interval $I\subset\Omega\backslash\{F=0\}$ there is $C>0$ such that
\[\int_{[0,\infty)}\chi_{\varphi(t,I)}(y)d\lambda(t)\leq C\mbox{ and }\int_{[0,\infty)}\chi_I(\varphi(t,y))d\lambda(t)\leq C\]
for every $y\in\Omega$. 
\end{lem}

\begin{proof}
Let $I=[a,b]$. Because $\partial_2\varphi(t,x)>0$ by (\ref{positive}) we have $\varphi(t,I)=[\varphi(t,a),\varphi(t,b)]$. It is easy to check that
\[\chi_{\varphi(t,I)}(y)=\chi_{(-\infty,y]}(\varphi(t,a))\chi_{[y,\infty)}(\varphi(t,b))\]
so that
\begin{equation}\label{integral1}
\int_{[0,\infty)}\chi_{\varphi(t,I)}(y)d\lambda(t)=\int_{[0,\infty)}\chi_{(-\infty,y]}(\varphi(t,a))\chi_{[y,\infty)}(\varphi(t,b))d\lambda(t).
\end{equation}
We now consider the case when $F_{|I}>0$. Then there is a unique $s>0$ such that $\varphi(s,a)=b$ so that with (\ref{integral1}) and $0<F(\varphi(t,a))=\partial_1\varphi(t,a)$ we obtain by an application of the Transformation Formula for Lebesgue integrals
\begin{eqnarray}\label{integral2}
\int_{[0,\infty)}\chi_{\varphi(t,I)}(y)d\lambda(t)&=&\int_{[0,\infty)}\chi_{(-\infty,y]}(\varphi(t,a))\chi_{[y,\infty)}(\varphi(s,\varphi(t,a)))d\lambda(t)\nonumber\\
&=&\int_{[0,\infty)}\chi_{(-\infty,y]}(\varphi(t,a))\chi_{[y,\infty)}(\varphi(s,\varphi(t,a)))\frac{F(\varphi(t,a))}{F(\varphi(t,a))}d\lambda(t)\\
&=&\int_{[a,\omega)}\chi_{(-\infty,y]}(r)\chi_{[y,\infty)}(\varphi(s,r))\frac{1}{F(r)}d\lambda(r),\nonumber
\end{eqnarray}
where $\omega=\lim_{t\rightarrow\infty}\varphi(t,a)$. Clearly, if $y\notin [a,\omega)$ we have
\[\chi_{(-\infty,y]}(r)\chi_{[y,\infty)}(\varphi(s,r))=0,\]
so that without loss of generality we may assume $y\in [a,\omega)$. If $y\leq b$ we obviously have $\chi_{(-\infty,y]}(r)\leq\chi_{(-\infty,b]}(r)$ so that we further obtain
\begin{equation}\label{integral3}
\int_{[a,\omega)}\chi_{(-\infty,y]}(r)\chi_{[y,\infty)}(\varphi(s,r))\frac{1}{F(r)}d\lambda(r)\leq\int_{[a,b]}\frac{1}{F(r)}d\lambda(r).
\end{equation}
Now, if $y>b$ we derive from $y\in [a,\omega)$ and $b=\varphi(s,a)$ that for some $u>s$ we have $y=\varphi(u,a)=\varphi(s,\varphi(u-s,a))$, in particular $y\in\varphi(s,\Omega)$. Taking into account that $\varphi(-s,y)\geq\varphi(-s,b)=a$ this gives
\begin{eqnarray}\label{integral4}
\int_{[a,\omega)}\chi_{(-\infty,y]}(r)\chi_{[y,\infty)}(\varphi(s,r))\frac{1}{F(r)}d\lambda(r)&=&\int_{[a,\omega)}\chi_{(-\infty,y]}(r)\chi_{[\varphi(-s,y),\infty)}(r)\frac{1}{F(r)}d\lambda(r)\nonumber\\
&=&\int_{[\varphi(-s,y),y]}\frac{1}{F(r)}d\lambda(r)\nonumber\\
&=&\int_{\varphi(\cdot,y)([-s,0])}\frac{1}{F(r)}d\lambda(r)\\
&=&\int_{[-s,0]}\frac{1}{F(\varphi(r,y))}\partial_1\varphi(r,y)d\lambda(r)\nonumber\\
&=&s,\nonumber
\end{eqnarray}
where we invoked the Transformation Formula as well as $\partial_1\varphi(s,y)=F(\varphi(s,y))$ once more. Combining (\ref{integral1})-(\ref{integral4}) we obtain in case of $F_{|I}>0$
\[\int_{[0,\infty)}\chi_{\varphi(t,I)}(y)d\lambda(t)\leq\max\{\int_I\frac{1}{F(r)}d\lambda(r),s\}.\]

Let us now turn to the case $F_{|I}<0$. Then there is a unique $s>0$ with $\varphi(s,b)=a$. Repeating the above calculations mutatis mutandis gives
\[\int_{[0,\infty)}\chi_{\varphi(t,I)}(y)d\lambda(t)\leq\max\{\int_I\frac{1}{-F(r)}d\lambda(r),s\},\]
which proves that
\[\{\int_{[0,\infty)}\chi_{\varphi(t,I)}(y)d\lambda(t); y\in\Omega\}\]
is bounded above.

In order to prove that
\[\{\int_{[0,\infty)}\chi_I(\varphi(t,y))d\lambda(t); y\in\Omega\}\]
is bounded above as well, we observe that if $\varphi(t,y)\in I$ for some $t\geq 0$ then $y$ belongs to the connected component of $\Omega\backslash\{F=0\}$ containing $I$. Thus, unless $\chi_I(\varphi(t,y))=0$ for every $t\geq 0$ we have with $J=\{\varphi(t,y);\,t\geq 0\}$
\begin{eqnarray*}
\int_{[0,\infty)}\chi_I(\varphi(t,y))d\lambda(t)&=&\int_{[0,\infty)}\chi_I(\varphi(t,y))\frac{|\partial_1\varphi(t,y)|}{|F(\varphi(t,y))|}d\lambda(t)\\
&=&\int_J \chi_I(r)\frac{1}{|F(r)|}d\lambda(r)\\
&\leq&\int_I\frac{1}{|F(r)|}d\lambda(r).
\end{eqnarray*}
Since the last intergal above is independent of $y$ this proves the lemma.
\end{proof}

{\it Proof of Theorem \ref{equivalence}}. If $T_{F,h}$ satisfies the Frequent Hypercyclicity criterion $T_{F,h}(t_0)$ is chaotic for every $t_0>0$ by \cite[Proposition 2.6]{MaPe11}. Thus, we only have to show that i) implies ii).

So let $T_{F,h}$ be chaotic. It follows from \cite[Theorem 1]{Ka14} that $\lambda(\{F=0\})=0$ so that
\[X_0:=\mbox{span}\{\chi_I;\,I\subseteq\Omega\backslash\{F=0\}\mbox{ compact interval}\},\]
is a dense subspace of $L^p_\rho(\Omega)$.

We define for $x\in\Omega$, $p\geq 1$, and $t\geq 0$
\begin{eqnarray*}
	\rho_{t,p}(x)&=&\chi_{\varphi(t,\Omega)}(x)|h_t(\varphi(-t,x))|^p\exp(\int_0^{-t}F'(\varphi(s,x))ds)\,\rho(\varphi(-t,x))\\
	&=&\chi_{\varphi(t,\Omega)}(x)\exp(p\int_0^t\Re h(\varphi(s,\varphi(-t,x)))ds)\exp(\int_0^{-t}F'(\varphi(s,x))ds)\,\rho(\varphi(-t,x))
\end{eqnarray*}
as well as
\begin{eqnarray*}
	\rho_{-t,p}(x)&=&|h_t(x)|^{-p}\exp(\int_0^t F'(\varphi(s,x))ds)\,\rho(\varphi(t,x))\\
	&=&\exp(-p\int_0^t\Re h(\varphi(s,x))ds)\exp(\int_0^t F'(\varphi(s,x))ds)\,\rho(\varphi(t,x)).
\end{eqnarray*}
Then $\rho_{0,p}=\rho$ and $\rho_{t,p}\geq 0$ for every $t\in\R$, and for fixed $x\in\Omega$ the mapping $t\mapsto\rho_{t,p}(x)$ is Lebesgue measurable.

By \cite[Lemma 7]{Ka09}, for every $[a,b]\subseteq\Omega\backslash\{F=0\}$ there is $C>0$ such that
\begin{equation}\label{inequality}
\forall\,t\in\R, x\in [a,b]: \frac{1}{C}\rho_{t,p}(\alpha)\leq \rho_{t,p}(x)\leq C\rho_{t,p}(\beta),
\end{equation} 
where $\alpha:=a$ and $\beta:=b$ in case of $F_{|[a,b]}>0$, respecitvely $\alpha:=b$ and $\beta:=a$ in case of $F_{|[a,b]}<0$.

For a compact interval $I\subset\Omega\backslash\{F=0\}$ and $t\geq 0$ we set
\[S_t\chi_I:=\chi_{\varphi(t,\Omega)}(\cdot)h_t(\varphi(-t,\cdot))^{-1}\chi_I(\varphi(-t,\cdot)).\]
Since for every $x\in\varphi(t,\Omega)\backslash\{F=0\}$ we have $\partial_2\varphi(t,\varphi(-t,x))\partial_2\varphi(-t,x)=1$ and $\varphi(-t,x)>0$ it follows
\begin{eqnarray}\label{well-defined}
\|S_t\chi_I\|^p&=&\int_\Omega\chi_{\varphi(t,\Omega)}(x)|h_t(\varphi(-t,x))|^{-p}\chi_I(\varphi(-t,x))\rho(x)d\lambda(x)\nonumber\\
&=&\int_\Omega\chi_{\varphi(t,\Omega)}(\varphi(t,\varphi(-t,x))|h_t(\varphi(-t,x))|^{-p}\chi_I(\varphi(-t,x))\partial_2\varphi(t,\varphi(-t,x))\nonumber\\
&&\cdot\rho(\varphi(t,\varphi(-t,x))\partial_2\varphi(-t,x)d\lambda(x)\nonumber\\
&=&\int_\Omega |h_t(y)|^{-p}\chi_I(y)\partial_2(\varphi(t,y))\rho(\varphi(t,y))d\lambda(y)\\
&=&\int_\Omega |h_t(y)|^{-p}\chi_I(y)\exp(\int_0^tF'(\varphi(s,y)ds)\rho(\varphi(t,y))d\lambda(y)\nonumber\\
&=&\int_I\rho_{-t,p}(y)d\lambda(y)\nonumber\\
&\leq& C\rho_{-t,p}(\beta),\nonumber
\end{eqnarray}
where $C>0$ depends on $I$, and $\beta=\sup I$ if $F_{|I}>0$, respecitvely $\beta=\inf I$ if $F_{|I}<0$ by (\ref{inequality}). Thus, $S_t\chi_I\in L^p_\rho(\Omega)$. It is easy to see that $S_t\chi_I= S_t\chi_{I_0}+S_t\chi_{I\backslash I_0}$ for closed intervals $I_0\subseteq I\subseteq\Omega\backslash\{F=0\}$. Hence, we can extend $S_t$ to a linear mapping from $X_0$ into $L^p_\rho(\Omega)$.

Moreover, it is straight forward to check that for a closed and  bounded interval $I\subseteq\Omega\backslash\{F=0\}$ we have
\[\forall\,t>0:\,T_{F,h}(t)S_t\chi_I=\chi_I\]
as well as
\[\forall\,r>t>0:\,T_{F,h}(t)S_r\chi_I=S_{r-t}\chi_I,\]
so that by the linearity of $T_{F,h}(t)$ and $S_t$ on $X_0$ hypothesis i) of the Frequent Hypercyclicity criterion is satisfied.

In order to show that conditions ii) and iii) of the Frequent Hypercyclicity criterion also hold we observe that by the linearity of $T_{F,h}(t)$ and $S_t$ on $X_0$ we have to show the Pettis integrabiliy of
\[[0,\infty)\rightarrow L^p_\rho(\Omega),t\mapsto T_{F,h}(t)\chi_I\mbox{ and }[0,\infty)\rightarrow L^p_\rho(\Omega),t\mapsto S_t\chi_I\]
for every compact interval $I\subset\Omega\backslash\{F=0\}$. In order to do so, we distinguish two cases.

First we assume $p=1$. From (\ref{well-defined}) we then obtain
\[\int_{[0,\infty)}\|S_t\chi_I\|d\lambda(t)\leq C\int_{[0,\infty)}\rho_{-t,1}(\beta)d\lambda(t).\]
Since $T_{F,h}$ is chaotic, it follows from \cite[Theorem 1 and Remark 6]{Ka14} that the integral on the above right hand side is finite. Thus,
\[[0,\infty)\rightarrow L^1_\rho(\Omega),t\mapsto S_t\chi_I\]
is in particular Pettis integrable. Moreover, for a compact interval $I\subseteq\Omega\backslash\{F=0\}$ we derive
\begin{eqnarray*}
\|T_{F,h}(t)\chi_I\|&=&\int_\Omega |h_t(x)|\chi_I(\varphi(t,x))\rho(x)d\lambda(x)\\
&=&\int_\Omega |h_t(\varphi(-t,\varphi(t,x)))|\chi_I(\varphi(t,x))\rho(\varphi(-t,\varphi(t,x)))\\
&&\cdot\partial_2\varphi(-t,\varphi(t,x))\partial_2\varphi(t,x)d\lambda(x)\\
&=&\int_{\varphi(t,\Omega)}|h_t(\varphi(-t,x))|\chi_I(x)\partial_2\varphi(-t,x)\rho(\varphi(-t,x))d\lambda(x)\\
&=&\int_I\rho_{t,1}(x)d\lambda(x)\\
&\leq& C\rho_{t,1}(\beta),
\end{eqnarray*}
where $C>0$ depends on $I$ and where we used again $\partial_2\varphi(-t,\varphi(t,x))\partial_2\varphi(t,x)=1$ and (\ref{inequality}), and where again $\beta=\sup I$ in case of $F_{|I}>0$, respectively $\beta=\inf I$ in case of $F_{|I}<0$. Therefore, we obtain
\[\int_{[0,\infty)}\|T_{F,h}(t)\chi_I\|d\lambda(t)\leq C\int_{[0,\infty)}\rho_{t,1}(\beta)d\lambda(t).\]
Since $T_{F,h}$ is chaotic it follows from \cite[Theorem 1 and Remark 6]{Ka14} that the integral on the right hand side of the above inequality is finite so that
\[[0,\infty)\rightarrow L^1_\rho(\Omega),t\mapsto T_{F,h}(t)\chi_I\]
is in particular Pettis integrable. So in case $p=1$ the theorem is proved.

To finish the proof we now consider the case $p>1$. Let $q\in (1,\infty)$ be the exponent conjugate to $p$, i.e.\ $1/p+1/q=1$. We fix $g\in L^q_\rho(\Omega)$. For a compact interval $I\subset\Omega\backslash\{F=0\}$ it follows for $t>0$, applying the Transformation Formula again,
\begin{eqnarray*}
\langle g, S_t\chi_I\rangle&=&\int_\Omega g(x)\chi_{\varphi(t,\Omega)}(x)h_t(\varphi(-t,x))^{-1}\chi_I(\varphi(-t,x))\rho(x)d\lambda(x)\\
&=&\int_{\varphi(t,\Omega)}g(x)h_t(\varphi(-t,x))^{-1}\chi_I(\varphi(-t,x))\rho(x)d\lambda(x)\\
&=&\int_\Omega g(\varphi(t,x))h_t(x)^{-1}\chi_I(x)\rho(\varphi(t,x))|\partial_2\varphi(t,x)|d\lambda(x).
\end{eqnarray*}
Since $\partial_2\varphi(t,x)=\exp(\int_0^t F'(\varphi(s,x))ds)$ we obtain for some constant $C>0$ by an application of Fubini's Theorem and H\"older's Inequality
\begin{eqnarray}\label{Pettis1}
\int_{[0,\infty)}|\langle g, S_t\chi_I\rangle|d\lambda(t)&\leq&\int_{[0,\infty)}\int_I|g(\varphi(t,x))||h_t(x)^{-1}|\rho(\varphi(t,x))\partial_2\varphi(t,x)d\lambda(x)d\lambda(t)\nonumber\\
&=&\int_I\int_{[0,\infty)}|g(\varphi(t,x))||h_t(x)^{-1}|\rho(\varphi(t,x))\partial_2\varphi(t,x)d\lambda(t)d\lambda(x)\nonumber\\
&\leq&\int_I\Big(\int_{[0,\infty)}|g(\varphi(t,x))|^q\rho(\varphi(t,x))\partial_2\varphi(t,x)d\lambda(t)\Big)^{\frac{1}{q}}\nonumber\\
&&\cdot\Big(\int_{[0,\infty)}|h_t(x)|^{-p}\rho(\varphi(t,x))\partial_2\varphi(t,x))d\lambda(t)\Big)^{\frac{1}{p}}d\lambda(x)\nonumber\\
&=&\int_I\Big(\int_{[0,\infty)}|g(\varphi(t,x))|^q\rho(\varphi(t,x))\partial_2\varphi(t,x)d\lambda(t)\Big)^{\frac{1}{q}}\\
&&\cdot\Big(\int_{[0,\infty)}\rho_{-t,p}(x)d\lambda(t)\Big)^{\frac{1}{p}}d\lambda(x)\nonumber\\
&\leq&\int_I\Big(\int_{[0,\infty)}|g(\varphi(t,x))|^q\rho(\varphi(t,x))\partial_2\varphi(t,x)d\lambda(t)\Big)^{\frac{1}{q}}d\lambda(x)\nonumber\\
&&\cdot C^{\frac{1}{p}}\Big(\int_{[0,\infty)}\rho_{-t,p}(\beta)d\lambda(t)\Big)^{\frac{1}{p}},\nonumber
\end{eqnarray}
where we used (\ref{inequality}) in the last step with $\beta=\sup I$ in case of $F_{|I}>0$, respectively $\beta=\inf I$ in case of $F_{|I}<0$. Since $T_{F,h}$ is chaotic, we have by \cite[Theorem 1 and Remark 6]{Ka14}
\begin{equation}\label{Pettis2}
\int_{[0,\infty)}\rho_{-t,p}(\beta)d\lambda(t)<\infty.
\end{equation}
Using H\"older's Inequality, Fubini's Theorem, and the Transformation Formula for Lebesgue integrals we further have
\begin{eqnarray}\label{Pettis3}
&&\int_I\Big(\int_{[0,\infty)}|g(\varphi(t,x))|^q\rho(\varphi(t,x))\partial_2\varphi(t,x)d\lambda(t)\Big)^{\frac{1}{q}}d\lambda(x)\nonumber\\
&\leq&\lambda(I)^{\frac{1}{p}}\Big(\int_I\int_{[0,\infty)}|g(\varphi(t,x))|^q\rho(\varphi(t,x))\partial_2\varphi(t,x)d\lambda(t)d\lambda(x)\Big)^{\frac{1}{q}}\nonumber\\
&=&\lambda(I)^{\frac{1}{p}}\Big(\int_{[0,\infty)}\int_{\varphi(t,I)}|g(x)|^q\rho(x)d\lambda(x)d\lambda(t)\Big)^{\frac{1}{q}}\\
&=&\lambda(I)^{\frac{1}{p}}\Big(\int_\Omega\int_{[0,\infty)}\chi_{\varphi(t,I)}(x)d\lambda(t)\;|g(x)|^q\rho(x)d\lambda(x)\Big)^{\frac{1}{q}}\nonumber\\
&\leq& \lambda(I)^{\frac{1}{p}}C^{\frac{1}{q}}\Big(\int_\Omega|g(x)|^q\rho(x)d\lambda(x)\Big)^{\frac{1}{q}},\nonumber
\end{eqnarray}
for some constant $C>0$ by Lemma \ref{aux}. Combining (\ref{Pettis1}), (\ref{Pettis2}), and (\ref{Pettis3}) we conclude
\[\forall\,g\in L^q_\rho(\Omega):\,[0,\infty)\rightarrow\C,t\mapsto\langle g,S_t \chi_I\rangle\mbox{ is integrable}\]
which imlpies by \cite[Theorem 4.5]{MaPe11} the Pettis integrability of $t\mapsto S_t\chi_I$.

Finally, let $g\in L^q_\rho(\Omega)$ and let $I\subset\Omega\backslash\{F=0\}$ be a compact interval. As above, one derives for $t\geq 0$
\[\langle g,T_{F,h}(t)\chi_I\rangle=\int_I\chi_{\varphi(t,\Omega)}(x)g(\varphi(-t,x))h_t(\varphi(-t,x))\rho(\varphi(-t,x))\partial_2\varphi(-t,x)d\lambda(x),\]
so that because $\partial_2\varphi(-t,x)=\exp(-\int_{-t}^0F'(\varphi(s,x))ds)$ for $x\in\varphi(t,\Omega)$ by the same kind of arguments as above
\begin{eqnarray}\label{Pettis4}
\int_{[0,\infty)}|\langle g,T_{F,h}\chi_I\rangle|d\lambda(t)&\leq&\int_I\Big(\int_{[0,\infty)}\chi_{\varphi(t,\Omega)}(x)|g(\varphi(-t,x))|^q\partial_2\varphi(-t,x)\rho(\varphi(-t,x))d\lambda(t)\Big)^{\frac{1}{q}}\nonumber\\
&&\cdot\Big(\int_{[0,\infty)}\chi_{\varphi(t,\Omega)}(x)|h_t(\varphi(-t,x))|^p\partial_2\varphi(-t,x)\rho(\varphi(-t,x))d\lambda(x)\Big)^{\frac{1}{p}}d\lambda(x)\nonumber\\
&=&\int_I\Big(\int_{[0,\infty)}\chi_{\varphi(t,\Omega)}(x)|g(\varphi(-t,x))|^q\partial_2\varphi(-t,x)\rho(\varphi(-t,x))d\lambda(t)\Big)^{\frac{1}{q}}\\
&&\cdot\Big(\int_{[0,\infty)}\rho_{t,p}(x)d\lambda(t)\Big)^{\frac{1}{p}}d\lambda(x)\nonumber\\
&\leq&\int_I\Big(\int_{[0,\infty)}\chi_{\varphi(t,\Omega)}(x)|g(\varphi(-t,x))|^q\partial_2\varphi(-t,x)\rho(\varphi(-t,x))d\lambda(t)\Big)^{\frac{1}{q}}\nonumber\\
&&\cdot C^\frac{1}{p}\Big(\int_{[0,\infty)}\rho_{t,p}(\beta)\Big)^{\frac{1}{p}}\nonumber,
\end{eqnarray}
where we again used (\ref{inequality}) in the last step with $\beta=\sup I$ in case of $F_{|I}>0$, respectively $\beta=\inf I$ in case of $F_{|I}<0$. Since $T_{F,h}$ is chaotic, we have by \cite[Theorem 1 and Remark 6]{Ka14}
\begin{equation}\label{Pettis5}
\int_{[0,\infty)}\rho_{t,p}(\beta)d\lambda(t)<\infty.
\end{equation}
Continuing as above
\begin{eqnarray*}
&&\int_I\Big(\int_{[0,\infty)}\chi_{\varphi(t,\Omega)}(x)|g(\varphi(-t,x))|^q\partial_2\varphi(-t,x)\rho(\varphi(-t,x))d\lambda(t)\Big)^{\frac{1}{q}}\\
&\leq&\lambda(I)^{\frac{1}{p}}\Big(\int_{[0,\infty)}\int_{\varphi(t,\Omega)}\chi_I(\varphi(t,\varphi(-t,x)))|g(\varphi(-t,x))|^q\rho(\varphi(-t,x))\partial_2\varphi(-t,x)d\lambda(x)d\lambda(t)\Big)^{\frac{1}{p}}\\
&=&\lambda(I)^{\frac{1}{p}}\Big(\int_\Omega\int_{[0,\infty)}\chi_I(\varphi(t,x))d\lambda(t)\,|g(x)|^q\rho(x)d\lambda(x)\Big)^{\frac{1}{q}}\\
&\leq&\lambda(I)^{\frac{1}{p}}C^{\frac{1}{p}}\Big(\int_\Omega|g(x)|^q\rho(x)d\lambda(x)\Big)^{\frac{1}{q}},
\end{eqnarray*}
for some suitable constant $C>0$ by Lemma \ref{aux}. Combining this with (\ref{Pettis4}) and (\ref{Pettis5}) gives
\[\forall\,g\in L^q_\rho(\Omega):\,[0,\infty)\rightarrow\C,t\mapsto\langle g,T_{F,h}(t) \chi_I\rangle\mbox{ is integrable.}\]
Thus, again by \cite[Theorem 4.5]{MaPe11} $t\mapsto T_{F,h}(t)\chi_I$ is Pettis integrable which proves the theorem.\hfill$\square$

\begin{example}
\begin{rm}
We consider $\Omega=(0,1)$ and $F(x)=-x$. Then, $(0,1)$ is forward invariant under $F$ and $\{F=0\}=\{x\in\Omega; F(x)=0\}=\emptyset$. For every $h\in C[0,1]$ for which
\begin{equation}\label{vFL}
[0,1]\rightarrow\C,x\mapsto\frac{h(x)- \Re h(0)}{x}\in L^1[0,1]
\end{equation}
it follows that hypotheses a) and b) of Theorem \ref{equivalence} are satisfied. Moreover, since $F'$ and $\Re h$ are bounded, $\rho=1$ is $p$-admissible for $F$ and $h$, so that $T_{F,h}$ is a well-defined $C_0$-semigroup on $L^p(0,1)$. Its generator is given by
\[Af(x)=-xf'(x)+h(x)f(x)\]
with domain
\[D(A)=\{f\in L^p(0,1);x\mapsto -xf'(x)+h(x)f(x)\in L^p(0,1)\},\]
where the derivative is to be understood in the distributional sense (see e.g.\ \cite[Theorem 15]{ArKaMa13}).

The first order partial differential equation
\[\frac{\partial}{\partial t}u(t,x)=-x\frac{\partial}{\partial x}u(t,x)+h(x)u(t,x),\,t>0, x\in (0,1)\]
with initial condition
\[u(0,x)=v(x),\,x\in (0,1)\]
is called the (linear) von Foerster-Lasota equation. For the corresponding $C_0$-semigroup the following are equivalent.
\begin{itemize}
	\item[i)] $\Re h(0)>-\frac{1}{p}$,
	\item[ii)] $T_{F,h}$ is hypercyclic on $L^p(0,1)$,
	\item[iii)] $T_{F,h}$ is chaotic on $L^p(0,1)$,
	\item[iv)] $T_{F,h}$ is frequently hypercyclic on $L^p(0,1)$.
\end{itemize}
Indeed, (\ref{vFL}) implies the equivalences of i) to iii) (cf.\ \cite[Theorem 27]{ArKaMa13}) while iv) follows from iii) by Theorem \ref{equivalence} and iv) trivially implies ii).

For real valued $h$ satisfying (\ref{vFL}) it was shown by Dawidowicz and Poskrobko in \cite{DaPo} that $T_{F,h}$ is strongly stable whenever $h(0)\leq -1/p$. Thus, there is a very strong dichotomie in the dynamical behavior of the von Foerster-Lasota semigroup on $L^p(0,1)$.

\end{rm}
\end{example}

\section{Weighted composition $C_0$-semigroups on Sobolev spaces}

Given a bounded interval $(a,b)$ and $F\in C^1[a,b]$ with $F(a)=0$ such that $(a,b)$ is forward invariant under $F$. Moreover, let $h\in W^{1,\infty}[a,b]$ be such that
\begin{itemize}
	\item[1)] $\forall\,x\in\{F=0\}:\,h(x)=h(a)\in\R$,
	\item[2)] the function $[a,b]\rightarrow\R, y\mapsto\frac{h(y)-h(a)}{F(y)}$ belongs to $L^\infty[a,b]$.
\end{itemize}
In \cite{ArKaMa13} it is shown that then the operator
\[A_p:\{f\in W^{1,p}[a,b];\, Ff''\in L^p[a,b]\}\rightarrow W^{1,p}[a,b], A_pf=Ff'+hf,\]
where the derivatives are taken in the distributional sense, generates a $C_0$-semigroup $S_{F,h}$ on $W^{1,p}[a,b]\, (1\leq p<\infty)$ given by
\[\forall\,t\geq 0, f\in W^{1,p}[a,b]:\,S(t)f(x)=h_t(x)f(\varphi(t,x)).\]
By \cite{ArKaMa13} this $C_0$-semigroup $S_{F,h}$ is never hypercyclic on $W^{1,p}[a,b]$. In particular, $S_{F,h}$ cannot be frequently hypercyclic nor chaotic on $W^{1,p}[a,b]$. 

Since $F(a)=0$, the closed subspace
\[W^{1,p}_*[a,b]:=\{f\in W^{1,p}[a,b];\,f(a)=0\}\]
of $W^{1,p}[a,b]$ is invariant under $S_{F,h}$ such that the restriction of $S_{F,h}$ to $W^{1,p}_*[a,b]$ defines a
$C_0$-semigroup on $W^{1,p}_*[a,b]$ which we denote again by $S_{F,h}$. Its generator is given by
\[A_{p,*}:\{f\in W^{1,p}_*[a,b];\, Ff''\in L^p[a,b]\}\rightarrow W^{1,p}[a,b], A_{p,*}f=Ff'+hf,\]
see \cite{ArKaMa13}. By \cite[Proposition 18, Theorem 20, and Proposition 24]{ArKaMa13} the above $C_0$-semigroup $S_{F,h}$ on $W^{1,p}_*[a,b]$ is linearly conjugate to the $C_0$-semigroup $T_{F,F'+h(a)}$ on $L^p(a,b)$, i.e.\ there is a continuous linear bijection $\Phi:L^p(a,b)\rightarrow W^{1,p}_*[a,b]$ such that
\[\forall\,t\geq 0:\,S_{F,h}(t)=\Phi\circ T_{F,F'+h(a)}(t)\circ\Phi^{-1}.\]
Hence, we can apply the result of the previous section once we have proved the following proposition.

\begin{prop}\label{linearly conjugate}
Let $X_1, X_2$ be Banach spaces, $T_1,T_2$ be $C_0$-semigroups on $X_1$, resp.\ $X_2$ such that $T_1$ satisfies the Frequent Hypercyclicity criterion. Moreover, assume there is a continuous linear injective operator $\Phi:X_1\rightarrow X_2$ which has dense range and which satisfies
\[\forall\,t\geq 0: T_2(t)\circ\Phi=\Phi\circ T_1(t).\]
Then, $T_2$ satisfies the Frequent Hypercyclicity criterion.
\end{prop}

\begin{proof}
By hypothesis there are a dense subset $X_0\subseteq X_1$ and maps $S_t:X_0\rightarrow X$ such that $t\mapsto T_1(t)x$ as well as $t\mapsto S_t x$ are Pettis integrable in $[0,\infty)$ for all $x\in X_0$, and such that $T_1(t)S_tx=x$ and $T_1(t)S_rx=S_{r-t}x$ whenever $x\in X_0$ and $r>t>0$. 

Since $\Phi$ has dense range, $\Phi(X_0)$ is dense in $X_2$ and the injectivity of $\Phi$ ensures that for $t\geq 0$
\[\tilde{S}_t:\Phi(X_0)\rightarrow X_2,\Phi(x)\mapsto\Phi(S_t x)\]
is well-defined. Using the transposed $\Phi^t$ of $\Phi$ it is straight forward to show that for all $x\in X_0$ the mappings $t\mapsto T_2(t)(\Phi(x))=\Phi(T_1(t)x)$ as well as $t\mapsto \tilde{S}_t(\Phi (x))=\Phi (S_tx)$ are Pettis integrable in $[0,\infty)$ for every $x\in X_0$.

Trivially, $T_2(t)\tilde{S}_t(\Phi (x))=\Phi (x)$ and $T_2(t)\tilde{S}_r(\Phi(x))=\tilde{S}_{r-t}(\Phi(x))$ for every $x\in X_0$ and $r>t>0$ so that $T_2$ satisfies the Frequent Hypercyclicity criterion.
\end{proof}

\begin{theo}\label{Sobolev}
Let $(a,b)$ be a bounded interval, $F\in C^1[a,b]$ with $F(a)=0$ such that $(a,b)$ is forward invariant under $F$.
Moreover, let $h\in W^{1,\infty}[a,b]$ be such that
\begin{itemize}
	\item[1)] $\forall\, x\in\{F=0\}:\,h(x)=h(a)\in\R$,
	\item[2)] the function $[a,b]\rightarrow\C, y\mapsto\frac{h(y)-h(a)}{F(y)}$ belongs to $L^\infty[a,b]$.
	\end{itemize}
	Then, for the $C_0$-semigroup $S_{F,h}$ on $W^{1,p}_*[a,b]$ the following are equivalent.
	\begin{itemize}
		\item[i)] $S_{F,h}$ is chaotic.
		\item[ii)] $S_{F,h}$ satisfies the Frequent Hypercyclicity criterion.
	\end{itemize}
\end{theo}

\begin{proof}
As in the proof of Theorem \ref{FHC} we only have to show that i) implies ii). As mentioned before Propositon \ref{linearly conjugate}, under the hypothesis of the theorem, it is shown in \cite{ArKaMa13} that there is a continuous linear bijection $\Phi:L^p(a,b)\rightarrow W^{1,p}_*[a,b]$ such that
\[\forall\,t\geq 0:\,S_{F,h}(t)=\Phi\circ T_{F,F'+h(a)}(t)\circ\Phi^{-1}.\]
Now, if $S_{F,h}$ is chaotic, applying the Comparison Principle \cite{GEPe11} it follows that $T_{F,F'+h(a)}$ is chaotic on $L^p(a,b)$. Hence, by Theorem \ref{FHC}, $T_{F,F'+h(a)}$ satisfies the Frequent Hypercyclicity criterion. An application of Proposition \ref{linearly conjugate} therefore proves the theorem.
\end{proof}

\begin{example}
\begin{rm}
We consider again the linear von Foerster-Lasota equation, i.e.\
\[\frac{\partial}{\partial t}u(t,x)=-x\frac{\partial}{\partial x}u(t,x)+h(x)u(t,x),\,t>0, x\in (0,1)\]
with initial condition
\[u(0,x)=v(x),\,x\in (0,1).\]
Now, we also impose the boundary condition
\[u(t,0)=0,t>0.\]
If $h\in W^{1,p}[0,1]$ with $h(0)\in\R$ and $y\mapsto\frac{h(y)-h(0)}{y}\in L^{\infty}[0,1]$, we can invoke our $C_0$-semigroup $S_{-id,h}$ on $W^{1,p}_*[0,1]$ to obtain a solution to the above initial-boundary value problem, whenever $v\in W^{1,p}_*[0,1]$.

For the dynamical properties of this $C_0$-semigroup on $W^{1,p}_*[0,1]$ the following are equivalent.
\begin{itemize}
	\item[i)] $h(0)>1-\frac{1}{p}$,
	\item[ii)] $S_{-id,h}$ is hypercyclic on $W^{1,p}_*[0,1]$,
	\item[iii)] $S_{-id,h}$ is chaotic on $W^{1,p}_*[0,1]$,
	\item[iv)] $S_{-id,h}$ is frequently hypercyclic on $W^{1,p}_*[0,1]$.
\end{itemize}
Indeed, the equivalences of i) to iii) are shown in \cite[Theorem 27]{ArKaMa13}) while iv) follows from iii) by Theorem \ref{Sobolev} and iv) trivially implies ii).
\end{rm}
\end{example}

\end{document}